\documentclass[
english,
headings=small,
paper=letter,
abstracton,
draft,
]{scrartcl}

\usepackage[T1]{fontenc}
\usepackage[latin9]{inputenc}
\usepackage{amsthm}
\usepackage{amsmath}
\usepackage{amssymb}
\usepackage{setspace}
\usepackage[matrix,arrow]{xy}
\usepackage{hyperref,url}
\makeatletter
\numberwithin{equation}{section}
\numberwithin{figure}{section}
\theoremstyle{plain}
\newtheorem{thm}{\protect\theoremname}

\date{}

\makeatother
\usepackage[textsize=footnotesize,obeyDraft]{todonotes}     
\usepackage{babel}
\providecommand{\theoremname}{Theorem}

\DeclareMathOperator{\Rep}{\mathbf{Rep}}
\DeclareMathOperator{\Ind}{\mathrm{Ind}}
\DeclareMathOperator{\Hom}{\mathrm{Hom}}

\DeclareMathOperator{\ad}{\mathrm{ad}}
\begin{document}

\title{Genericity under parahoric restriction}

\author{
Manish Mishra\thanks{University of Heidelberg, Germany, electronic address: \url{manish.mishra@gmail.com}} 
\and
Mirko R\"osner\thanks{University of Heidelberg, Germany, electronic address: \url{mirko_rosner@hotmail.com}}
}





\maketitle
\begin{abstract}
We study the preservation of genericity under parahoric restriction of depth zero representations. 
\end{abstract}

\section{Introduction}

Let $\mathbf{G}$ be a connected reductive group defined over a non-archimedean
local field $k$.  Let $\mathbf{B}$ be a $k$-Borel subgroup of $\mathbf{G}$ with unipotent
radical $\mathbf{U}$ and let $\mathbf{T}$ be a maximal $k$-torus in $\mathbf{B}$. The corresponding groups of $k$-rational points are $G$, $B$, $U$, $T$. A character
$\psi:U\rightarrow\mathbb{C}^{\times}$ is called generic if the
stabilizer of $\psi$ in $T$ is exactly the center $Z$ of
$G$. An  admissible representation $\pi$ of $G$ is
called \textit{generic}\textbf{ }(more specifically $\psi$-\textit{generic})
if there exists a generic character $\psi$ of $U$ such that $\Hom_{G}(\pi,\Ind_{U}^{G}\psi)\neq0$. 

A basic result due to Rodier \cite[Thm.\,2]{Ro73} states that genericity is preserved under
the Jacquet functor. On the category of depth zero representations
of $G$ (see \cite{MP96} for the notion of depth), there is a functor analogous to the Jacquet functor called the \textit{parahoric restriction}
functor \cite{vig03}. It is defined as follows. Let $G_x$ be a parahoric subgroup
of $G$ with pro-unipotent radical $G_x^{+}$. The quotient $G_x/G_x^{+}$
is the $\mathbb{F}_{q}$-points of a connected reductive group $\underline{\mathbf{M}}$
defined over $\mathbb{F}_{q}$, where $\mathbb{F}_{q}$ is the residue
field of $k$. The parahoric restriction functor sends a representation
$(\pi,V)$ of $G$ to the representation of $\underline{\mathbf{M}}(\mathbb{F}_{q})$
obtained by restricting $\pi$ to $G_x$ and then taking the $G_x^{+}$-invariants
of it. 

%

We say $\mathbf{G}$ is \textit{unramified}, if it is quasisplit and splits over an unramified extension.
Suppose that $\mathbf{G}$ is unramified, that $G_x$ is contained in a hyperspecial parahoric subgroup of $G$
and that $\pi$ is parabolically induced from a supercuspidal representation of a Levi subgroup of $G$. Then
we show in Theorem \ref{thm:main} that the parahoric restriction functor preserves
genericity of $\pi$.
This does not generalize to arbitrary parahorics and admissible representations, see Section \ref{sec:example}.

\section{Notations}
Fix a non-archimedean local field $k$ and an unramified connected reductive group $\mathbf{G}$ over $k$.  Let $\mathbf{B}=\mathbf{T}\ltimes \mathbf{U}$ be a $k$-Borel subgroup of $\mathbf{G}$
with a maximal $k$-torus $\mathbf{T}$ and a unipotent
radical $\mathbf{U}$. Let $\mathbf{Z}$ be the center of $\mathbf{G}$ and let $\mathbf{T}_{\ad}:=\mathbf{T}/\mathbf{Z}$.
Their groups of $k$-rational points are $G$, $B$, $T$, $U$, $Z$ and $T_{\ad}$, respectively.
We will follow the standard abuse of notation to write ``parabolic subgroups
of $G$'' instead of ``$k$-points of $k$-parabolic subgroups of $\mathbf{G}$''. 

For a point $x$ in the Bruhat-Tits building of $G$, the associated parahoric subgroup will be denoted by $G_x$. Its Levi quotient is $\underline{G_x}=G_x/G_x^+$ with the pro-unipotent radical $G_x^+$.
We will denote by $\Rep(G)$ the category of admissible complex representations of $G$ and likewise for the other groups.
\section{Parahoric restriction functor}
Fix a point $x$ in the apartment attached to $T$. Restricting an admissible representation $(\pi,V)$ of $G$ to the parahoric $G_x$ and taking invariants with respect to the pro-unipotent radical $G_x^+$ gives rise to a representation of $\underline{G_x}=G_x/G_x^+$. This defines the parahoric restriction functor
\begin{equation}
\mathbf{r}_{G_x}:\Rep(G)\to\Rep(\underline{G_x})\,,
\qquad
\begin{cases}
  (\pi,V)\longmapsto \left(\underline{G_x}\to \mathrm{Aut}(V^{G_x^+})\right),\\
 (V_1\to V_2) \longmapsto  \left(V_1^{G_x^+}\to V_2^{G_x^+}\right),
 \end{cases}
\end{equation}
where $V_1\to V_2$ is a morphism between admissible representations $(\pi_1,V_1)$ and $(\pi_2,V_2)$. This functor is exact and defines a homomorphism between the corresponding Grothen\-dieck groups, analogous to Jacquet's functor of parabolic restriction.

For hyperspecial parahorics, parahoric restriction commutes with parabolic induction in the following sense:
\begin{thm}\label{thm:commutes_with_par_ind}
Let $x$ be a hyperspecial point in the apartment attached to $T$ with corresponding hyperspecial parahoric subgroup $G_x\subseteq G$. Fix a standard parabolic subgroup $P\supseteq B$ with Levi decomposition $P=M\ltimes N$ such that $T\subseteq M$. Then the following diagram is commutative up to natural equivalence:
\begin{align*}
\begin{xy}
\xymatrix{
\Rep(G)\ar[r]^-{\mathbf{r}_{G_x}}                    &      \Rep(\underline{G_x})             \\
\Rep(M)\ar[r]_-{\mathbf{r}_{M_x}}\ar[u]^{\Ind_P^G}   &      \Rep(\underline{M_x})\text{.}\ar[u]_{\Ind_{\underline{P}}^{\underline{G_x}}}
}
\end{xy} 
\end{align*}
\end{thm}
The parabolic subgroup $\underline{P}=(G_x\cap P)/(G_x^+\cap P)$ of $\underline{G_x}$ is generated by the same roots as $P$ and has Levi subgroup $\underline{M_x}$.

\begin{proof} 
The parahoric subgroup of $M$ attached to $x$ is $M_x\cong G_x\cap M$, cp.~\cite{MP96}.

Fix an admissible representation $(\sigma,V)$ of $M$ and denote its inflation to $P$ by the same symbol. We construct a natural equivalence $(\Ind_P^G(\sigma))^{G_x^+}\to \Ind_{\underline{P}}^{\underline{G_x}}(\mathbf{r}_{M_x}(\sigma))$. Then $\Ind_P^G(\sigma)$ has a canonical model by right multiplication on the space of functions $f:G\to V$ with
\begin{equation*}
 f(pg)=\delta_P^{1/2}(p)\sigma(p)f(g)\qquad \forall p\in P,\;\; g\in G,
\end{equation*} where $\delta_P$ is the modulus character. By the Iwasawa decomposition $G=PG_x$, every such $f$ is uniquely defined by its restriction $\tilde{f}=f|_{G_x}$ to $G_x$.  The linear map $f\mapsto \tilde{f}$ is thus a $G_x$-equivariant isomorphism from $\Ind_P^G(\sigma)$ to the space of $\tilde{f}:G_x\to V$ with 
\begin{align*}\tilde{f}(pg)=\sigma(p)\tilde{f}(g)\qquad \forall p\in G_x\cap P,\;\; g\in G_x.\end{align*}
Such an $\tilde{f}:G_x\to V$ is invariant under the right action of $G_x^+$ if and only if $\tilde{f}(xu)=\tilde{f}(x)$ for every $u\in G_x^+$ and $x\in G_x$. In that case we have \begin{align*}\sigma(p)\tilde{f}(g)=\tilde{f}(pg)=\tilde{f}(g g^{-1}pg)=\tilde{f}(g)\qquad\forall p\in P\cap G_x^+,\;\; g\in G_x,\end{align*} since $g^{-1}pg\in G_x^+$. Hence $\tilde{f}$ actually maps into the invariant space $V^{P\cap G_x^+}$.

Now $\tilde{f}$ factors over a unique function $h_f:\underline{G_x}\to V^{P\cap G_x^+}$ with the property
\begin{equation*}h_f(\underline{p}\underline{g})=\mathbf{r}_{M_x}(\sigma)(\underline{p})h_f(\underline{g})\qquad\forall \underline{p}\in\underline{P}\;\;\underline{g}\in\underline{G}.
\end{equation*}  The action of $\underline{G_x}$ by right-multiplication on the space of these $h_f$ is a model of the induced representation $\Ind_{\underline{P}}\circ \mathbf{r}_{M_x}(\sigma)$ of $\underline{G_x}$.
The family of isomorphisms $\{f\mapsto h_f\}_\sigma$ provides the natural equivalence.
\end{proof}

\section{Generic depth zero}
A character $\psi:U\rightarrow\mathbb{C}^{\times}$ is called \textit{generic} if its stabilizer in $T_{\ad}$ is
trivial. Let $z$ be a hyperspecial vertex of $G$ in the apartment attached to $T$. A generic character $\psi$ of $U$ is called \textit{depth-zero} at $z$ if its restriction to $U\cap G_z$ factors through a generic character $\psi_{z}$ of  $\underline{U}:=(U\cap G_z)/(U\cap G_z^+)$.

\begin{thm}\label{thm:main}
Let $\pi=\Ind_{P,M}^G(\sigma)$ be an irreducible admissible representation of $G$ that is parabolically induced from a supercuspidal irreducible representation $\sigma$ of a Levi subgroup $M$ of a parabolic $P\subseteq G$. Let $G_{x}$ be a parahoric subgroup of $G$, contained in a hyperspecial parahoric subgroup, such that $\underline{\pi}:=\mathbf{r}_{G_x}(\pi)\neq0$.  Then $\pi$ being generic implies that $\underline{\pi}$ is generic. 
\end{thm}

\begin{proof}
By conjugating $x$ and $P$ if necessary, we can assume without loss of generality that $P \supseteq B$, that $M \supseteq T$, that $x$ is a point in the apartment associated to $T$ and that $\pi$ is generic with respect to a generic character $\psi$ of $U$. We can assume without loss that $G_x$ is a hyperspecial maximal parahoric subgroup of $G$ because of transitivity of parahoric restriction \cite[4.1.3]{vig03} and Rodier's result \cite[Thm.\,2]{Ro73}.

If $w_{o}$ is an element in the normalizer of
$T$ such that $B\cap w_{o}Bw_{o}^{-1}=T$, then $Q:=M\cap w_{o}Uw_{o}^{-1}$
is a maximal unipotent subgroup of $M$ \cite[Thm.\,2]{Ro73}. Define a generic character of $Q$ by
$\psi_{M}(q)=\psi(w_o^{-1}qw_o)$ for $q\in Q$.
Then by \cite[Thm.\,2]{Ro73},
\begin{equation}
\Hom_{U}(\pi,\psi)\cong\Hom_{Q}(\sigma,\psi_{M}).
\end{equation}
 Since $\underbar{\ensuremath{\pi}}\neq 0$, Theorem \ref{thm:commutes_with_par_ind} implies that $\sigma$ is a depth zero supercuspidal representation of $M$. By \cite[Lemma 6.1.2]{DeRe09}, there is a hyperspecial point $y$ of $M$ and a cuspidal representation $\tau^{\circ}$ of $\underline{M_y}$ such that 
\begin{enumerate}
\item[a)] $\psi_{M}$ is depth zero at $y$ and $\tau^{\circ}$ is $(\psi_{M})_{y}$-generic.
\item[b)] There is an extension of $\tau^{\circ}$ to a representation $\tau$ of the normalizer $[M_y]$ of $M_y$ in $M$ such that $\sigma$=c-$\Ind_{[M_y]}^M\tau$. Note that since $y$ is a hyperspecial point of $M$, $[M_y]=Z_MM_y$, where $Z_M$ denotes the center of $M$.

\end{enumerate}
Since $\sigma$ has depth zero at $x$ by Thm.~\ref{thm:commutes_with_par_ind}, we can assume without loss of generality that $x=y$ (see proof of \cite[Lemma 3.3(ii)]{Yu01}). We have therefore
\begin{equation}
\Hom_{\underline{Q}}(\tau^{\circ},(\psi_{M})_{x})\neq0,
\end{equation}
where $\underline{Q}$ denotes a maximal unipotent subgroup of 
$\underline{M_{x}}$ defined in the same way as $Q$. Theorem \ref{thm:commutes_with_par_ind} and a result of Vign\'{e}ras \cite[\S7]{Vigneras_Barcelona} imply that $\underline{\pi}$ is isomorphic to $\Ind_{\underline{P}}^{\underline{G_x}}\tau^{\circ}$.
Then again by \cite[Thm.\,2]{Ro73},
\begin{equation}
\Hom_{\underline{U}}(\underline{\pi},\psi_{x})\cong\Hom_{\underline{Q}}(\tau^{\circ},(\psi_{M})_{x})\neq0.
\end{equation}
 This completes the proof.
\end{proof}

\section{Non-special parahoric restriction}\label{sec:example}
In Theorem~\ref{thm:main} we make two technical assumptions: we assume that $\pi$ is parabolically induced and that $G_x$ is contained in a hyperspecial maximal parahoric subgroup. However, if we drop these assumptions, there is a counterexample.


\textit{Counterexample:} 
There is a generic irreducible admissible representation $\pi$ of $G=\mathrm{GSp}(4,k)$ and a parahoric subgroup $G_x\subseteq G$ such that the parahoric restriction $\mathbf{r}_{G_x}(\pi)$ is non-zero, but not generic.
\begin{proof}
Let $\xi$ be the non-trivial unramified quadratic character of $k^\times$. Let $G=\mathrm{GSp}(4,k)$ be the group of symplectic similitudes with respect to the symplectic form $\left(\begin{smallmatrix}&w\\-w&\end{smallmatrix}\right)$ for $w=\left(\begin{smallmatrix}&1\\1&\end{smallmatrix}\right)$. Fix the Borel pair $(B,T)$ where $B=T\ltimes U\subseteq G$ is the subgroup of upper triangular matrices and $T\subseteq B$ is the maximal torus of diagonal matrices. Fix a character
\begin{align*}
\alpha:T\to\mathbb{C}^\times,\quad \mathrm{diag}(a,b,c/a,c/b)\mapsto \xi(ab)\,|a|\,|c|^{-1/2}.
\end{align*}
Inflating $\alpha$ to $B$ gives rise to the induced representation $\Ind_B^G \alpha$, which admits a unique irreducible subrepresentation\footnote{This is type Va in the notation of Roberts and Schmidt \cite{Roberts-Schmidt}.} $\pi$. The representation $\pi$ is generic for arbitrary with respect to every generic character $\psi$ of $G$.

The standard paramodular subgroup $G_x\subseteq G$ is a non-special maximal parahoric subgroup.
The second author has shown \cite[Thm.\,3.7]{Roesner_thesis} that the parahoric restriction $\mathbf{r}_{G_x}(\pi)$ is isomorphic to the restriction of $\mathbf{1}\otimes\mathrm{St}\oplus \mathrm{St} \otimes \mathbf{1}$ to \begin{equation*}\{(\xi_1,\xi_2)\in(\mathrm{GL}(2,\mathbb{F}_q))^2\,|\,\det \xi_1=\det \xi_2\}\cong \underline{G_x}.\end{equation*} where $\mathbf{1}$ is the trivial and $\mathrm{St}$ is the Steinberg representation of $\mathrm{GL}(2,\mathbb{F}_q)$. Thus $\mathbf{r}_{G_x}(\pi)$ is not generic.
\end{proof}

\section*{Acknowledgment}
The authors are thankful to Sandeep Varma for pointing out a mistake in the earlier draft of this article.

\bibliographystyle{plain}
\bibliography{p_res}

\end{document}